\documentclass[a4paper, 11pt, reqno]{amsart}
\usepackage{amsmath,amsfonts,amssymb,amsthm,enumerate}
\usepackage{graphicx}
\usepackage{xy} \xyoption{all}
\usepackage{xcolor}

\newtheorem{thm}{Theorem}[section]
\newtheorem{cor}[thm]{Corollary}
\newtheorem{prop}[thm]{Proposition}
\newtheorem{lem}[thm]{Lemma}

\theoremstyle{definition}
\newtheorem{defn}[thm]{Definition}
\newtheorem{nota}[thm]{Notation}
\newtheorem{exas}[thm]{Example}

\let\phi\varphi

\begin{document}
\title{Simple Lie algebras arising from Steinberg algebras of Hausdorff ample groupoids}
\maketitle
\begin{center}
Tran Giang Nam\footnote{Institute of Mathematics, VAST, 18 Hoang Quoc Viet, Cau Giay, Hanoi, Vietnam. E-mail address: \texttt{tgnam@math.ac.vn}} 

\end{center}

\begin{abstract} In this paper, we show that a unital simple Steinberg algebra is central, and a nonunital simple Steinberg algebra has zero center. We identify the fields $K$ and Hausdorff ample groupoids $\mathcal{G}$ for which the simple Steinberg algebra $A_K(\mathcal{G})$ yields a simple Lie algebra $[A_K(\mathcal{G}), A_K(\mathcal{G})]$. We apply the obtained results on simple Leavitt path algebras, simple Kumjian-Pask algebras and simple Exel-Pardo algebras to determine their associated Lie algebras are simple. In particular, we give easily computable criteria to determine which Lie algebras of
the form $[L_K(E), L_K(E)]$ are simple, when $E$ is an arbitrary graph and the Leavitt path algebra $L_K(E)$ is simple. Also, we obtain that unital simple Exel-Pardo algebras are central,  and nonunital simple  Exel-Pardo algebras have zero center.
\medskip


  
\end{abstract}

\section{Introduction}
Steinberg algebras were introduced in \cite{stein:agatdisa} in the context of discrete inverse semigroup algebras and independently in \cite{cfst:aggolpa} as a model for Leavitt path algebras, which is a discrete analogue of groupoid $C^*$-algebras
(\cite{exel:isacca}, \cite{p:gisatoa} and \cite{r:agatca}). This class of algebras includes group algebras, inverse semigroup
algebras, Leavitt path algebras \cite{ap:tlpaoag05, amp:nktfga}, Kumjian-Pask algebras \cite{kp:hrgca, acar:kpaohrg}, and Exel-Pardo algebras \cite{exelpar:ssgautokanca, cp:kpaofahrg, exelparstar:caossgoag, hpss:aaaoepa}. 

With each associative $K$-algebra $R$ one may construct the \textit{Lie $K$-algebra} $[R, R]$ of $R$, consisting of all $K$-linear combinations of elements of the form $xy - yx$ where $x, y \in R$. Then $[R, R]$ becomes a Lie
algebra under the operation $[x, y] = xy - yx$ for all $x, y \in R.$ In particular, when $R$ is the Steinberg algebra $A_K(\mathcal{G})$ of a Hausdorff ample groupoid $\mathcal{G}$, one may construct and subsequently investigate the Lie algebra
$[A_K(\mathcal{G}), A_K(\mathcal{G})]$. Such an analysis was carried out in \cite{am:slaaflpa} in the case where $\mathcal{G}$ is the graph groupoid $\mathcal{G}_E$ associated to a row-finite graph $E$ for which the Leavitt path algebra $L_K(E)$ is simple.
In \cite[Corollaries 21 and 2.2, and Theorem 2.3]{am:slaaflpa}
easily computable necessary and sufficient conditions were given
which determine the simplicity of the Lie algebra $[L_K(E), L_K(E)]$ in this situation.

In this article, we show that a unital simple Steinberg algebra is central, and a nonunital simple Steinberg algebra has zero center (Theorem~\ref{centofsimSteinAlg}). We apply the result together with Herstein's result \cite[Theorem 1.13]{Her:tirt} (see Theorem~\ref{Herstein} below) in order to identify the fields $K$ and Hausdorff ample groupoids $\mathcal{G}$ for which the simple Steinberg algebra $A_K(\mathcal{G})$ yields a simple Lie algebra $[A_K(\mathcal{G}), A_K(\mathcal{G})]$ (Theorem~\ref{simpLieSteinalg}). Consequently, we obtain the following interesting results. Firstly, together with Abrams and Mesyan's result \cite[Theorem 23]{am:slaaflpa}, we give 
easily computable necessary and sufficient conditions (Theorems~\ref{simpLieLPAalg} and \ref{B-space3}) to determine which Lie algebras of the form $[L_K(E), L_K(E)]$ are simple, when $E$ is an arbitrary graph and the Leavitt path algebra $L_K(E)$ is simple, which generalize Abrams and Mesyan's main results in \cite[Section 3]{am:slaaflpa}. Secondly, we obtain that unital simple Kumjian-Pask algebras of row-finite $k$-graphs without sources are central, and nonunital simple  Kumjian-Pask algebras have zero center (Theorem~\ref{cenofsimKPalg}), which may recover Brown and an Huef's result \cite[Theorem 4.7]{ba:coaathrg}, as well as give criteria (Theorem~\ref{simpLieKPalg}) to determine which Lie algebras of the form $[KP_K(\Lambda), KP_K(\Lambda)]$ are simple, when $\Lambda$ is a row-finite $k$-graph without sources and the Kumjian-Pask algebra $KP_K(\Lambda)$ is simple. Thirdly, we completely describe the center of simple Exel-Pardo algebras (Theorem~\ref{cenofsimEPalg}) which shows that unital simple Exel-Pardo algebras are central,  and nonunital simple  Exel-Pardo algebras have zero center, and give criteria for Lie algebras of the form $[L_K(G, E), L_K(G, E)]$ are simple (Theorem~\ref{simLieEPalg}), when the Exel-Pardo algebra $L_K(G, E)$ is simple.
\medskip\medskip

We now present a streamlined version of the necessary background ideas. Given a ring $R$ and two elements $x, y\in R$, we let $[x, y]$ denote the commutator $xy - yx$, and let $[R, R]$ denote the additive subgroup of $R$ generated by the commutators. Then $[R, R]$ is a Lie ring, with operation $x\ast y = [x, y] = xy-yx$, which we call the \textit{Lie ring associated} to $R$.  If $R$ is in addition an algebra over a field $K$, then $[R, R]$ is s $K$-subspace of $R$, and in this way becomes a Lie $K$-algebra, which we call the \textit{Lie $K$-algebra associated} to $R$. Clearly $[R, R] = 0$ if and only if $R$ is commutative.

For a $d\times d$ matrix $A\in M_d(R)$, $\rm{trace}(A)$ denotes as usual the sum of the diagonal entries of $A$. We will utilize the following fact about traces.

\begin{prop}[{\cite[Corollary 17]{mes:cr}}]\label{trace}
Let $R$ be a unital ring, $d$ a positive integer, and $A\in M_d(R)$. Then $A\in [M_d(R), M_d(R)]$ if and only if $\rm{trace}(A)\in [R, R]$. In particular, any $A\in M_d(R)$ of trace zero is necessarily in $[M_d(R), M_d(R)]$.
\end{prop}

Let $L$ denote a Lie ring (respectively, Lie $K$-algebra). A subset $I$ of $L$ is called a \textit{Lie ideal} if $I$ is an additive subgroup (respectively, $K$-subspace) of $L$ and $[L, I] \subseteq I.$  The Lie ring (respectively, Lie $K$-algebra) $L$ is called \textit{simple} if $[L, L] \neq 0$ and the only Lie ideals of $L$ are $0$ and $L$. It is well-known \cite[Lemma 2]{am:slaaflpa} (see also \cite[Page 34]{jj:lrodoar}) that a Lie $K$-algebra $L$ over a field $K$ is simple as a Lie ring if and only if $L$ is simple as a Lie $K$-algebra. As a consequence of this note, throughout the article we will often use the concise phrase ``$L$ is simple" to indicate that the Lie $K$-algebra $L$ is simple either as a Lie ring or as a Lie $K$-algebra. The following result of Herstein will play a pivotal role
in our analysis.

\begin{thm}[{\cite[Theorem 1.13]{Her:tirt}}]\label{Herstein}
Let $S$ be a simple ring. Assume that either $\rm{char}(S)\neq 2$ or that $S$ is not $4$-dimensional over $Z(S)$, where $Z(S)$ is a field. Then $U \subseteq Z(S)$ for any proper Lie ideal $U$ of the Lie ring $[S, S]$.
\end{thm}

\section{Simplicity of the Lie algebra of a simple Steinberg algebra}
In this section, we show that a unital simple Steinberg algebra is central, and a nonunital simple Steinberg algebra has zero center (Theorem~\ref{centofsimSteinAlg}). We apply the result together with Herstein's result cited above in order to identify the fields $K$ and Hausdorff ample groupoids $\mathcal{G}$ for which the simple Steinberg algebra $A_K(\mathcal{G})$ yields a simple Lie algebra $[A_K(\mathcal{G}), A_K(\mathcal{G})]$ (Theorem~\ref{simpLieSteinalg}). All these constructions are crucially based on some general notions of groupoids that for the reader's convenience we reproduce here.

A \textit{groupoid} is a small category in which every morphism is invertible. It can also be viewed as a generalization of a group which has a partial binary operation. 
Let $\mathcal{G}$ be a groupoid. If $x \in \mathcal{G}$, $s(x) = x^{-1}x$ is the source of $x$ and $r(x) = xx^{-1}$ is its range. The pair $(x,y)$ is is composable if and only if $r(y) = s(x)$. The set $\mathcal{G}^{(0)}:= s(\mathcal{G}) = r(\mathcal{G})$ is called the \textit{unit space} of $\mathcal{G}$. Elements of $\mathcal{G}^{(0)}$ are units in the sense that $xs(x) = x$ and
$r(x)x = x$ for all $x\in \mathcal{G}$. For $U, V \subseteq \mathcal{G}$, we define
\begin{center}
$UV = \{\alpha\beta \mid \alpha\in U, \beta\in V \text{\, and \,} r(\beta) = s(\alpha)\}$ and $U^{-1} = \{\alpha^{-1}\mid \alpha\in U\}$.
\end{center}

A \textit{topological groupoid} is a groupoid endowed with a topology under which the inverse map is continuous, and such that the composition is continuous with respect to the relative topology on $\mathcal{G}^{(2)} := \{(\beta, \gamma)\in \mathcal{G}^2\mid s(\beta) = r(\gamma)\}$ inherited from $\mathcal{G}^2$. An \textit{\'{e}tale groupoid} is a topological groupoid $\mathcal{G},$ whose unit space  $\mathcal{G}^{(0)}$ is locally compact Hausdorff, and such that the domain map $s$ is a local homeomorphism. In this case, the range map $r$ and the multiplication map are local homeomorphisms and  $\mathcal{G}^{(0)}$ is open in $\mathcal{G}$ \cite{r:egatq}. 

An \textit{open bisection} of $\mathcal{G}$ is an open subset $U\subseteq \mathcal{G}$ such that $s|_U$ and $r|_U$ are homeomorphisms onto an open subset of $\mathcal{G}^{(0)}$. Similar to \cite[Proposition 2.2.4]{p:gisatoa} we have that $UV$ and $U^{-1}$ are compact open bisections for all compact open bisections $U$ and $V$ of an  \'{e}tale groupoid  $\mathcal{G}$.
If in addition $\mathcal{G}$  is Hausdorff, then $U\cap V$ is a compact open bisection (see \cite[Lemma 1]{r:tgatlpa}). An \'{e}tale groupoid $\mathcal{G}$ is called \textit{ample} if  $\mathcal{G}$ has a base of compact open bisections for its topology. 

Steinberg algebras were introduced in \cite{stein:agatdisa} in the context of discrete inverse semigroup algebras and independently in \cite{cfst:aggolpa} as a model for Leavitt path algebras. We recall the notion of the Steinberg algebra
as a universal algebra generated by certain compact open subsets of a Hausdorff ample groupoid.

\begin{defn}[{The Steinberg algebra}] 
Let $\mathcal{G}$ be a Hausdorff ample groupoid, and $K$ a field. The \textit{Steinberg algebra} of $\mathcal{G}$, denoted  $A_K(\mathcal{G})$,  is the algebra generated by the set
$\{t_U\mid U$ is a compact open bisection of $\mathcal{G}\}$ with coefficients in $K$,  subject to
\begin{itemize}
\item[(R1)] $t_{\emptyset} = 0$;
	
\item[(R2)] $t_Ut_V = t_{UV}$ for all compact open bisections $U$ and $V$; and
	
\item[(R3)] $t_U + t_V = t_{U\cup V}$ whenever $U$ and $V$ are disjoint compact open bisections such that $U\cup V$ is a bisection.
\end{itemize}
\end{defn}
Every element $f\in A_K(\mathcal{G})$ can be expressed as $f = \sum_{U\in F}k_Ut_U$, where $F$ is a finite set of compact open bisections. It was proved in \cite[Theorem 3.10]{cfst:aggolpa} that the Steinberg algebra defined above is isomorphic to the following construction:

\begin{center}
$A_K(\mathcal{G})= \{1_U\mid U$ is a compact open bisection of $\mathcal{G}\}$,
\end{center}
where $1_U: \mathcal{G}\longrightarrow K$ denotes the characteristic function on $U$. Equivalently, if we give $K$ the discrete topology, then $A_K(\mathcal{G}) = C_c(\mathcal{G}, K)$, the space of compactly supported continuous functions from $\mathcal{G}$ to $K$. Addition is point-wise and multiplication is given by convolution

\[(f\ast g)(\gamma)= \sum_{\gamma = \alpha\beta}f(\alpha)g(\beta)\] for all $\gamma\in \mathcal{G}$. It is useful to note that \[1_U\ast 1_V = 1_{UV}\] for compact open bisections $U$ and $V$ (see \cite[Proposition 4.5]{stein:agatdisa}) and the isomorphism between the two constructions is given by $t_U\longmapsto 1_U$ on the generators.
By \cite[Lemma 3.5]{cfst:aggolpa}, every element $f\in A_K(\mathcal{G})$ can be expressed as $f = \sum_{U\in F}k_U1_U$, where $F$ is a finite set of  mutually disjoint compact open bisections. By \cite[Proposition~4.11]{stein:agatdisa}, $A_K(\mathcal{G})$ is unital if and only if $\mathcal{G}^{(0)}$ is compact. In this case, the identity element $1 = 1_{\mathcal{G}^{(0)}}$.

Let $\mathcal{G}$ be a groupoid and $D, E$ subsets of $\mathcal{G}^{(0)}$. Define
\begin{center}
$\mathcal{G}_D := \{\gamma \in \mathcal{G}\mid s(\gamma)\in D \}$, $\mathcal{G}^E := \{\gamma \in \mathcal{G}\mid r(\gamma)\in E\}$ and $\mathcal{G}^E_D := \mathcal{G}^E\cap \mathcal{G}_D$.
\end{center}
In a slight abuse of notation, for $u, v\in \mathcal{G}^{(0)}$ we denote $\mathcal{G}_u := \mathcal{G}_{\{u\}}$, $\mathcal{G}^v := \mathcal{G}^{\{v\}}$ and $\mathcal{G}^v_u := \mathcal{G}^v\cap \mathcal{G}_u$. For a unit $u$ of $\mathcal{G}$ the group $\mathcal{G}^u_u = \{\gamma \in \mathcal{G}\mid u= r(\gamma) = s(\gamma)\}$ is called its \textit{isotropy group}. The \textit{isotropy group} of $\mathcal{G}$ is $\text{Iso}(\mathcal{G}): = \cup_{u\in \mathcal{G}^{(0)}}\mathcal{G}^u_u$. A subset $D$ of $\mathcal{G}^{(0)}$ is called \textit{invariant} if $s(\gamma)\in D$ implies $r(\gamma)\in D$ for all $\gamma\in \mathcal{G}$. Equivalently, $D = \{r(\gamma) \mid s(\gamma)\in D\} = \{s(\gamma) \mid r(\gamma)\in D\}$. Also, $D$ is invariant if and only if its complement is invariant.

\begin{defn}[{\cite[Definition 2.1]{bcfs:soaateg}}]
Let $\mathcal{G}$ be a Hausdorff ample groupoid. We say that $\mathcal{G}$ is \textit{minimal} if $\mathcal{G}^{(0)}$ has no nontrivial open invariant subsets. We say that $\mathcal{G}$ is \textit{effective} if the interior of $\text{Iso}(\mathcal{G})\setminus \mathcal{G}^{(0)}$ is empty.
\end{defn}

The following theorem provides us with a criterion for Steinberg algebras of Hausdorff ample groupoids to be simple, which plays an important role in the current note.

\begin{thm}[{\cite[Theorem 4.1]{bcfs:soaateg} and \cite[Theorem 4.1]{ce:utfsa}}]\label{SimpSteinAlg}
The Steinberg algebra  $A_K(\mathcal{G})$ of  a Hausdorff ample groupoid $\mathcal{G}$ over a field $K$ is simple if and only if $\mathcal{G}$ is both effective and minimal.		
\end{thm}

The center of the Steinberg algebra $A_K(\mathcal{G})$, denoted $Z(A_K(\mathcal{G}))$, is characterized in \cite[Proposition~4.13]{stein:agatdisa}. First, we say $f\in A_K(\mathcal{G})$ is a \textit{ class function} if $f$ satisfies the following conditions:

\begin{itemize}
\item[(1)] $f(\alpha)\neq 0$ implies $r(\alpha) = s(\alpha)$;
	
\item[(2)] $s(\alpha) = r(\alpha) = s(\beta)$ implies $f(\alpha) = s(\beta\alpha \beta^{-1})$.
\end{itemize}
Then \cite[Proposition~4.13]{stein:agatdisa} says that  the center of $A_K(\mathcal{G})$ is
\begin{center}
$Z(A_K(\mathcal{G})) = \{f\in A_K(\mathcal{G})\mid f \text{ is a class function}\}$.	
\end{center}
Using this note we next computer the center of simple Steinberg algebras. To do so, we need to describe various elements of a Steinberg algebra which are in its center.

\begin{lem}\label{classfun}
Let $K$ be a field, $\mathcal{G}$ a Hausdorff ample groupoid and $U$ a compact open invariant subset of $\mathcal{G}^{(0)}$. Then $1_U$ is a class function.	
\end{lem}
\begin{proof} Since $U$ is a compact open subset of $\mathcal{G}^{(0)}$ and $\mathcal{G}^{(0)}$ is Hausdorff, $U$ is clopen in $\mathcal{G}^{(0)}$. Since $\mathcal{G}^{(0)}$ is clopen in $\mathcal{G}$ by the Hausdorff property, $U$ is clopen in $\mathcal{G}$. This implies that $1_U\in A_K(\mathcal{G})$.
Let $\alpha\in \mathcal{G}$ with $1_U(\alpha) \neq 0$. We then have $\alpha \in U$, and so $s(\alpha) = r(\alpha)$.

Let $\alpha$ and $\beta\in \mathcal{G}$ with $s(\alpha) = r(\alpha) = s(\beta)$. If $\alpha\in U$, then $\beta\alpha \beta^{-1} = \beta\beta^{-1}$ and $s(\beta) = s(\alpha) = \alpha$. Since $U$ is an invariant subset of $\mathcal{G}^{(0)}$, $r(\beta) = \beta \beta^{-1}\in U$, and so $\beta\alpha \beta^{-1}\in U$. This implies that $1_U(\alpha) = 1 = 1_U(\beta\alpha \beta^{-1})$.

Consider the case $\alpha\notin U$. Assume that $\beta\alpha \beta^{-1} = u \in U$. We then have $r(\beta) = u$ and
$$s(\beta)=\beta^{-1} \beta =\beta^{-1}u \beta = \beta^{-1}\beta\alpha \beta^{-1}\beta = s(\beta)\alpha s(\beta) = r(\alpha)\alpha s(\alpha) = \alpha.$$ Since $U$ is an invariant subset of $\mathcal{G}^{(0)}$ and $r(\beta) =u\in U$, we must have  $\alpha = s(\beta) \in U$, a contradiction. This shows that $\beta\alpha \beta^{-1}\notin U$, and so $1_U(\beta\alpha \beta^{-1}) = 0 = 1_U(\alpha)$. Therefore, $1_U$ is a class function, thus finishing the proof.
\end{proof}

\begin{lem}\label{openvarsupport}
Let $K$ be a field and $\mathcal{G}$ a Hausdorff ample groupoid, and let  $f\in Z(A_K(\mathcal{G}))$. Then $\rm{supp}(f)\cap \mathcal{G}^{(0)}$ is a compact open invariant set.
\end{lem}
\begin{proof}
We first have that $f(\mathcal{G})\setminus\{0\}$ is contained in a compact subset of the discrete space $K$, and so it is finite. Assume that $f(\mathcal{G})\setminus\{0\} = \{k_1, \hdots, k_n\}$.	For each $1\le i\le n$, we have that $f^{-1}(k_i)$ is a clopen subset of $\mathcal{G}$. Since $f^{-1}(k_i) \subseteq \rm{supp}(f)$ and $\rm{supp}(f)$ is compact, $f^{-1}(k_i)$ is compact. Since $\mathcal{G}^{(0)}$ is clopen in $\mathcal{G}$ by the Hausdorff property, $f^{-1}(k_i) \cap \mathcal{G}^{(0)}$ is a compact open subset of $\mathcal{G}$.

We next claim that $f^{-1}(k_i) \cap \mathcal{G}^{(0)}$ is invariant. Indeed, let $\alpha\in \mathcal{G}$ with $s(\alpha) = \alpha^{-1}\alpha\in f^{-1}(k_i) \cap \mathcal{G}^{(0)}$. We then have $f(\alpha^{-1}\alpha) = k_i$ and $\alpha\alpha^{-1} = \alpha (\alpha^{-1}\alpha)\alpha^{-1}$. Since $f\in Z(A_K(\mathcal{G}))$ and by \cite[Proposition~4.13]{stein:agatdisa}, $f$ is a class function, and so $f(\alpha\alpha^{-1}) = f(\alpha (\alpha^{-1}\alpha)\alpha^{-1}) = f(\alpha^{-1}\alpha) = k_i$, that means, $r(\alpha) = \alpha\alpha^{-1}\in f^{-1}(k_i) \cap \mathcal{G}^{(0)}$. Therefore, $f^{-1}(k_i) \cap \mathcal{G}^{(0)}$ is invariant, showing the claim.

Thus we obtain that $f^{-1}(k_i) \cap \mathcal{G}^{(0)}$ is a  compact open invariant subset of $\mathcal{G}^{(0)}$ for all $i$, and so $\rm{supp}(f)\cap \mathcal{G}^{(0)} = \cup^n_{i=1}(f^{-1}(k_i) \cap \mathcal{G}^{(0)})$ is a compact open invariant set, finishing the proof.
\end{proof}

The following result provides us with a complete description of the center of a simple Steinberg algebra.

\begin{thm}\label{centofsimSteinAlg}
Let $K$ be a field and $\mathcal{G}$ a Hausdorff ample groupoid for which $A_K(\mathcal{G})$ is a simple Steinberg algebra. Then the following holds:

$(1)$ If $A_K(\mathcal{G})$	is unital, then $Z(A_K(\mathcal{G})) = K\cdot 1_{\mathcal{G}^{(0)}}$.

$(2)$ If $A_K(\mathcal{G})$	is not unital, then $Z(A_K(\mathcal{G})) = 0$.
\end{thm}
\begin{proof}
(1) Since $A_K(\mathcal{G})$ is unital and by \cite[Proposition 4.11]{stein:agatdisa}, $\mathcal{G}^{(0)}$ is a compact subset of $\mathcal{G}$ and $1_{\mathcal{G}^{(0)}}$ is the identity of $A_K(\mathcal{G})$. Then, by Lemma~\ref{classfun}, $1_{\mathcal{G}^{(0)}}$ is a class function, and so $1_{\mathcal{G}^{(0)}}\in Z(A_K(\mathcal{G}))$.

Let $f$ be a nonzero element of $Z(A_K(\mathcal{G}))$. We then have that $f(\mathcal{G})\setminus\{0\}$ is contained in a compact subset of the discrete space $K$, and so it is finite. Assume that $f(\mathcal{G})\setminus\{0\} = \{k_1, \hdots, k_n\} \subseteq K\setminus\{0\}$. We have that each $U_i := f^{-1}(k_i)$ is a clopen subset $\mathcal{G}$ which is contained in $\rm{supp}(f)$. Since $\rm{supp}(f)$ is compact, $U_i$ is compact open in $\mathcal{G}$. It is obvious that $U_i\cap U_j = \varnothing$ for all $1\le i\neq j\le n$ and
$f = k_11_{U_1} + \cdots + k_n 1_{U_n}$.

We next claim that each $1_{U_i}$ is a class function. Indeed, let $\alpha\in \mathcal{G}$ with $1_{U_i}(\alpha) \neq 0$. We then have that $\alpha \in U_i$ and $f(\alpha) = k_i 1_{U_i}(\alpha) = k_i\neq 0$. Since $f$ is a nonzero element of $Z(A_K(\mathcal{G}))$ and by \cite[Proposition 4.13]{stein:agatdisa}, $f$ is a class function, and so $s(\alpha) = r(\alpha)$. Let $\alpha$ and $\beta\in \mathcal{G}$ with $s(\alpha) = r(\alpha) = s(\beta)$. Then, since $f$ is a class function, $f(\alpha) = f(\beta\alpha\beta^{-1})$. We note that since $U_1, \hdots, U_n $ are mutually disjoint compact open subsets of $\mathcal{G}$, $f(\gamma) = k_i$ if and only if $\gamma\in U_i$ for all $\gamma \in \mathcal{G}$. From those observations we obtain that
\begin{center}
$\alpha\in U_i \Longleftrightarrow k_i = k_i\cdot 1_{U_i}(\alpha) = f(\alpha) = f(\beta\alpha\beta^{-1})\Longleftrightarrow\beta\alpha\beta^{-1}\in U_i$,	
\end{center}
so $1_{U_i}(\alpha) = 1_{U_i}(\beta\alpha\beta^{-1})$, showing the claim, that is, $1_{U_i}\in Z(A_K(\mathcal{G}))$ for all $i$.

Using Lemma~\ref{centofsimSteinAlg} we have that $U_i \cap \mathcal{G}^{(0)} = \rm{supp}(1_{U_i})\cap \mathcal{G}^{(0)}$ is a compact open invariant set. Since $A_K(\mathcal{G})$ is simple and by Theorem~\ref{SimpSteinAlg}, $\mathcal{G}$ is minimal, and so we have either $U_i \cap \mathcal{G}^{(0)} = \varnothing$ or $U_i \cap \mathcal{G}^{(0)} = \mathcal{G}^{(0)}$ for all $i$. If there exists an element $1\le i\le n$ such that $U_i \cap \mathcal{G}^{(0)} = \varnothing$, then $U_i \subseteq \mathcal{G}\setminus \mathcal{G}^{(0)}$. Since $\mathcal{G}$ is ample and $U_i$ is open in $\mathcal{G}$, there exists a compact open bisection $B$ of $\mathcal{G}$ such that $B\subseteq U_i\subseteq\mathcal{G}\setminus \mathcal{G}^{(0)}$. Since $A_K(\mathcal{G})$ is simple and by Theorem~\ref{SimpSteinAlg}, $\mathcal{G}$ is effective, and so there exists an element $\gamma\in B$ such that $s(\gamma) \neq r(\gamma)$. On the other hand, Since $B\subseteq U_i= \rm{supp}(1_{U_i})$ and $1_{U_i}$ is a class function (by the above claim), we must have $s(\lambda) = r(\lambda$ for all $\lambda\in B$, in particular $s(\gamma) = r(\gamma)$, a contradiction. Hence we have $U_i \cap \mathcal{G}^{(0)} = \mathcal{G}^{(0)}$ for all $i$, that means, $\mathcal{G}^{(0)} \subseteq U_i$ for all $i$. Then, since $U_1, \hdots, U_n $ are mutually disjoint, we must obtain that $n =1$, that is, $f = k_1 1_{U_1}$. This implies that $U_1$ is a compact open subset of $\mathcal{G}$ such that $\mathcal{G}^{(0)} \subseteq U_1$ and $s(\alpha) = r(\alpha)$ for all $\alpha\in U_1$, and so $U_1 \subseteq \text{Iso}(\mathcal{G})$. Since $A_K(\mathcal{G})$ is simple and by Theorem~\ref{SimpSteinAlg}, $\mathcal{G}$ is effective, and so, by \cite[Lemma 3.1]{bcfs:soaateg}, the interior of $\text{Iso}(\mathcal{G})$, denoted $\mathring{\text{Iso}(\mathcal{G})}$, is $\mathcal{G}^{(0)}$. Since $U_1$ is open in $\mathcal{G}$ and $\mathcal{G}^{(0)}\subseteq U_1 \subseteq \text{Iso}(\mathcal{G})$, we have $\mathcal{G}^{(0)} \subseteq \mathring{U_1} = U_1 \subseteq \mathring{\text{Iso}(\mathcal{G})} = \mathcal{G}^{(0)}$, and so $U_1 = \mathcal{G}^{(0)}$ and $f = k_1 1_{\mathcal{G}^{(0)}}$. Thus we obtain that $Z(A_K(\mathcal{G})) = K\cdot 1_{\mathcal{G}^{(0)}}$.

(2) Assume that $A_K(\mathcal{G})$ is not unital and $Z(A_K(\mathcal{G}))$ contains a nonzero element $f$. In a similar way as it was done in item (1) we have $f = k1_U$ for some $k\in K\setminus\{0\}$ and for some compact open subset $U$ of $\mathcal{G}$. Since $f\in Z(A_K(\mathcal{G}))$ and $U = \rm{supp}(f)$, $s(\alpha) = r(\alpha)$ for all $\alpha\in U$, that is, $U \subseteq \text{Iso}(\mathcal{G})$.  Since $A_K(\mathcal{G})$ is simple and by Theorem~\ref{SimpSteinAlg}, $\mathcal{G}$ is effective, and so, by \cite[Lemma 3.1]{bcfs:soaateg}, the interior of $\text{Iso}(\mathcal{G})$, denoted $\mathring{\text{Iso}(\mathcal{G})}$, is $\mathcal{G}^{(0)}$. We then have  $\mathcal{G}^{(0)} \subseteq \mathring{U} = U \subseteq \mathring{\text{Iso}(\mathcal{G})} = \mathcal{G}^{(0)}$, and so $\mathcal{G}^{(0)} = U$. This implies that $\mathcal{G}^{(0)}$ is compact, and hence $A_K(\mathcal{G})$ is unital by \cite[Proposition 4.11]{stein:agatdisa}, a contradiction. Thus we must have $Z(A_K(\mathcal{G})) = 0$, finishing the proof.
\end{proof}

Using Theorem~\ref{centofsimSteinAlg} and Herstein's result \cite[Theorem 1.13]{Her:tirt} we give criteria for the Lie algebra associated to a simple Steinberg algebra is simple. To do so, we need to establish some useful facts.

\begin{lem}\label{divisionSteinAlg}
Let $K$ be a field and $\mathcal{G}$ a Hausdorff ample groupoid for which $A_K(\mathcal{G})$ is a simple Steinberg algebra. Then
the following statements are equivalent:

$(1)$ $A_K(\mathcal{G})$ is a division ring;

$(2)$ $\mathcal{G}$ is a singleton;

$(3)$ $A_K(\mathcal{G})\cong K$.	
\end{lem}
\begin{proof}
(1)$\Longrightarrow$(2). Assume that $A_K(\mathcal{G})$ is a division ring. If $\mathcal{G}^{(0)}$	contains distinct elements $u$ and $v$, then since $\mathcal{G}$ is Hausdorff ample, there exist two compact open bisections $U$ and $V$ of $\mathcal{G}$ such that $u\in U$, $v\in V$ and $U\cap V = \varnothing$. Since $\mathcal{G}^{(0)}$ is clopen in $\mathcal{G}$ by the Hausdorff property, $W_1:=U \cap \mathcal{G}^{(0)}$ and $W_2:=V\cap \mathcal{G}^{(0)}$ are compact open subsets of $\mathcal{G}^{(0)}$ satisfying that $u\in W_1$, $v\in W_2$ and $W_1\cap W_2 = \varnothing$. Then $1_{W_1}\ast 1_{W_2} = 1_{W_1\cap W_2} = 1_{\varnothing} = 0$, and so $A_K(\mathcal{G})$ is not a division ring, a contradiction. Therefore, $\mathcal{G}^{(0)}$	is a singleton, and so $\mathcal{G}$ is exactly a group with discrete topology. We then have that $\text{Iso}(\mathcal{G}) = \mathcal{G}$. Since $A_K(\mathcal{G})$ is simple and by Theorem~\ref{SimpSteinAlg}, $\mathcal{G}$ is effective, and so the interior of $\text{Iso}(\mathcal{G})$ is $\mathcal{G}^{(0)}$, that means, $\mathcal{G} = \mathcal{G}^{(0)}$. This implies that $\mathcal{G}$ is a singleton.

The directions of (2)$\Longrightarrow$(3) and (3)$\Longrightarrow$(1) are obvious, finishing the proof.
\end{proof}

We say a simple Steinberg algebra $A_K(\mathcal{G})$ \textit{nontrivial} in case $A_K(\mathcal{G})\ncong K$.

\begin{lem}\label{nonzerocentofSteinAlg}
Let $K$ be a field, $\mathcal{G}$ an effective Hausdorff ample groupoid and $R= A_K(\mathcal{G})$. If $[R, R]\neq 0$, then $[[R, R], [R, R]]\neq 0$. In particular, if $R$ is a nontrivial simple Steinberg algebra, then $[[R, R], [R, R]]\neq 0$.
\end{lem}
\begin{proof}
If $s(\alpha) = r(\alpha)$ for all $\alpha\in \mathcal{G}$, then we first have $\mathcal{G} = \text{Iso}(\mathcal{G})$. Since $\mathcal{G}$ is effective, the interior of $\text{Iso}(\mathcal{G})$, denoted $\mathring{\text{Iso}(\mathcal{G})}$, is $\mathcal{G}^{(0)}$. This implies that $\mathcal{G} = \mathring{\mathcal{G}} = \mathring{\text{Iso}(\mathcal{G})} = \mathcal{G}^{(0)}$. We note that $1_U\ast 1_V = 1_{U\cap V}$ for compact open subsets $U$ and $V$ of $\mathcal{G}^{(0)}$ (by \cite[Proposition 4.5 (3)]{stein:agatdisa}). From those observations we obtain that $R = A_K(\mathcal{G}^{(0)})$ is commutative, and so $[R, R]=0$, a contradiction. Therefore, there exists an element $\alpha\in \mathcal{G}$ such that $s(\alpha)\neq r(\alpha)$. Since $\mathcal{G}$ is a Hausdorff ample groupoid, there exist three compact open bisections $B$, $V_1$ and $V_2$ of $\mathcal{G}$ such that $\alpha\in B$, $s(\alpha)\in V_1$, $r(\alpha)\in V_2$ and $V_1\cap V_2 = \varnothing$. Since $\mathcal{G}^{(0)}$ is clopen in $\mathcal{G}$ by the Hausdorff property, $W_1:=V_1 \cap \mathcal{G}^{(0)}$ and $W_2:=V_2\cap \mathcal{G}^{(0)}$ are compact open subsets of $\mathcal{G}^{(0)}$. We note that $s(\alpha)\in W_1$, $r(\alpha)\in W_2$ and $W_1\cap W_2 = \varnothing$. Let $U := W_2BW_1$. We then have that $U$ is a compact open bisection of $\mathcal{G}$ by \cite[Proposition 2.2.4]{p:gisatoa} (see, also \cite[Lemma 1]{r:tgatlpa}), and so 
$s(U)$ and $r(U)$ are compact open subsets of $\mathcal{G}^{(0)}$, and $U^{-1}:= \{\gamma^{-1}\mid \gamma\in U\}$ is a compact open bisection of $\mathcal{G}$.
It is clear that $\alpha\in U$, $s(U)\subseteq W_1$, $r(U)\subseteq W_2$ and $s(U)\cap r(U) = \varnothing$. This implies that $1_{r(U)}\ast 1_{s(U)} = 1_{s(U)}\ast 1_{r(U)} = 1_{s(U)\cap r(U)} = 1_{\varnothing} = 0$, and so $1_U\ast 1_{r(U)} = 1_{Us(U)}\ast 1_{r(U)}= 1_U\ast 1_{s(U)}\ast 1_{r(U)}
=0$ and $1_{r(U)}\ast 1_{U^{-1}}=1_{r(U)}\ast 1_{s(U)U^{-1}}= 1_{r(U)}\ast 1_{s(U)}\ast 1_{U^{-1}} =0$. We then have
$$[1_{r(U)}, 1_U] = 1_{r(U)}\ast 1_U - 1_U\ast 1_{r(U)}= 1_{r(U)U}= 1_U$$ and $$[1_{U^{-1}}, 1_{r(U)}] = 1_{U^{-1}}\ast 1_{r(U)} - 1_{r(U)}\ast 1_{U^{-1}}= 1_{U^{-1}r(U)}= 1_{U^{-1}},$$ so $$[[1_{r(U)}, 1_U], [1_{U^{-1}}, 1_{r(U)}]] = [1_U, 1_{U^{-1}}] =  1_{r(U)}-1_{s(U)}\neq 0.$$ This shows that $[[R, R], [R, R]]\neq 0$.

Assume that $R$ is a nontrivial simple Steinberg algebra. If $[R, R] = 0$, then we have that $R$ is commutative, and so $Z(R) = R \neq 0$. Then, since $R$ is simple and by Theorem~\ref{centofsimSteinAlg}, $Z(R) = R= K\cdot 1_R\cong K$, that means, $R$ is a trivial simple Steinberg algebra, a contradiction. This implies that $[R, R]\neq 0$, and so $[[R, R], [R, R]]\neq 0$, thus finishing the proof.
\end{proof}

We are now in position to provide the main result of this section.

\begin{thm}\label{simpLieSteinalg}
Let $K$ be a field and $\mathcal{G}$ a Hausdorff ample groupoid for which $A_K(\mathcal{G})$ is a simple Steinberg algebra. Then the following holds:
	
$(1)$ If $\mathcal{G}^{(0)}$ is not compact, then $[A_K(\mathcal{G}), A_K(\mathcal{G})]$ is a simple Lie $K$-algebra.
	
$(2)$ If $\mathcal{G}^{(0)}$ is compact and $A_K(\mathcal{G})$	is nontrivial, then $[A_K(\mathcal{G}), A_K(\mathcal{G})]$ is a simple Lie $K$-algebra if and only if $1_{\mathcal{G}^{0}}\notin [A_K(\mathcal{G}), A_K(\mathcal{G})]$.
\end{thm}
\begin{proof}
(1) Assume that $\mathcal{G}^{(0)}$ is not compact. By \cite[Proposition 4.11]{stein:agatdisa},
$A_K(\mathcal{G})$ is not unital. Then, by Theorem~\ref{centofsimSteinAlg} (2), $Z(A_K(\mathcal{G})) = 0$. Since $A_K(\mathcal{G})$ is simple and by \cite[Corollary 4]{am:slaaflpa}, we have either the Lie $K$-algebra $[A_K(\mathcal{G}), A_K(\mathcal{G})]$ is simple, or $[[ A_K(\mathcal{G}),  A_K(\mathcal{G})], [ A_K(\mathcal{G}),  A_K(\mathcal{G})]]= 0$. If $A_K(\mathcal{G})\cong K$ as $K$-algebras, then  $ A_K(\mathcal{G})$ is unital, contradicting with our hypothesis. Therefore, $ A_K(\mathcal{G})$ is a nontrivial simple Steinberg algebra, and so $[[ A_K(\mathcal{G}),  A_K(\mathcal{G})], [ A_K(\mathcal{G}),  A_K(\mathcal{G})]]\neq 0$ by Lemma~\ref{nonzerocentofSteinAlg}. This implies that 
$[A_K(\mathcal{G}), A_K(\mathcal{G})]$ is a simple Lie $K$-algebra.

(2) Assume that $\mathcal{G}^{(0)}$ is compact and $A_K(\mathcal{G})$ is a nontrivial simple algebra. By \cite[Proposition 4.11]{stein:agatdisa}, $A_K(\mathcal{G})$ is unital and the identity of $A_K(\mathcal{G})$ is $1_{\mathcal{G}^{(0)}}$. If $1_{\mathcal{G}^{(0)}}\in [A_K(\mathcal{G}), A_K(\mathcal{G})]$, then the $K$-subspace $<1_{\mathcal{G}^{(0)}}>$ of $[A_K(\mathcal{G}), A_K(\mathcal{G})]$ generated by $1_{\mathcal{G}^{(0)}}$ is a nonzero Lie ideal of $[A_K(\mathcal{G}), A_K(\mathcal{G})]$. Assume that $<1_{\mathcal{G}^{(0)}}> = [A_K(\mathcal{G}), A_K(\mathcal{G})]$.
Then, since $<1_{\mathcal{G}^{(0)}}>$ is a commutative subalgebra of $A_K(\mathcal{G})$, we immediately obtain that $[[ A_K(\mathcal{G}),  A_K(\mathcal{G})], [A_K(\mathcal{G}),  A_K(\mathcal{G})]]= 0$. On the other hand, since $A_K(\mathcal{G})$ is a nontrivial simple algebra and by Lemma~\ref{nonzerocentofSteinAlg}, $[[ A_K(\mathcal{G}),  A_K(\mathcal{G})], [A_K(\mathcal{G}),  A_K(\mathcal{G})]]\neq 0$, a contradiction. This implies that $<1_{\mathcal{G}^{(0)}}>$ is proper, and so the Lie $K$-algebra $[A_K(\mathcal{G}), A_K(\mathcal{G})]$ is not simple.

Conversely, if $1_{\mathcal{G}^{(0)}}\notin [A_K(\mathcal{G}), A_K(\mathcal{G})]$, then $Z(A_K(\mathcal{G}))\cap [A_K(\mathcal{G}), A_K(\mathcal{G})] = 0$ by Theorem~\ref{centofsimSteinAlg} (1). Since $A_K(\mathcal{G})$ is nontrivial simple and by Lemma~\ref{nonzerocentofSteinAlg}, $[[ A_K(\mathcal{G}),  A_K(\mathcal{G})], [A_K(\mathcal{G}),  A_K(\mathcal{G})]]\neq 0$. Assume that  $\rm{char}(K) =2$ and $A_K(\mathcal{G})$ is $4$-dimentional over $Z(A_K(\mathcal{G}))\cong K$. Then, we first note that it is well-known that a $4$-dimensional center simple $K$-algebra is either isomorphic to $M_2(K)$, or a division ring, and so we have either $A_K(\mathcal{G})\cong M_2(K)$ or $A_K(\mathcal{G})$ is a division ring. If $A_K(\mathcal{G})$ is a division ring, then by Lemma~\ref{divisionSteinAlg}, $A_K(\mathcal{G})\cong K$, contradicting with our hypothesis. Therefore, we must obtain that $A_K(\mathcal{G})\cong M_2(K)$. But, since $\rm{char}(K) =2$, $1 \in [M_2(K), M_2(K)]$ by Proposition~\ref{trace}, contradicting with our hypothesis. Thus, we have either $\rm{char}(K) \neq2$, or $A_K(\mathcal{G})$ is not $4$-dimensional over $Z(A_K(\mathcal{G}))\cong K$, and so the desired conclusion now follows from Theorem~\ref{Herstein}, thus finishing the proof.
\end{proof}

\section{Application: Leavitt path algebras}
In this section, based on Theorem~\ref{simpLieSteinalg}, we give 
easily computable necessary and sufficient conditions (Theorems~\ref{simpLieLPAalg} and \ref{B-space3}) to determine which Lie algebras of the form $[L_K(E), L_K(E)]$ are simple, when $E$ is an arbitrary graph and the Leavitt path algebra $L_K(E)$ is simple, which generalize Abrams and Mesyan's results \cite[Section 3]{am:slaaflpa}. All these constructions are crucially based on some general notions of graph theory, which for the reader’s convenience we reproduce here.


A (directed) graph $E = (E^0, E^1, s, r)$ consists of two disjoint sets $E^0$ and $E^1$, called \emph{vertices} and \emph{edges} respectively, together with two maps $s, r: E^1 \longrightarrow E^0$.  The vertices $s(e)$ and $r(e)$ are referred to as the \emph{source} and the \emph{range}
of the edge~$e$, respectively. A graph $E$ is called \emph{row-finite} if $|r^{-1}(v)|< \infty$ for all $v\in E^0$.  

A vertex~$v$ for which $r^{-1}(v)$ is empty is called a \emph{source}; a vertex~$v$ is \emph{regular} if $0< |r^{-1}(v)| < \infty$; and a vertex~$v$ is an \textit{infinite receiver} if $|r^{-1}(v)| = \infty$.

A \emph{path of length} $n$ in a graph $E$ is a sequence $p = e_{1} \cdots e_{n}$  of edges $e_{1}, \dots, e_{n}$ such that $s(e_{i}) = r(e_{i+1})$ for $i = 1, \dots, n-1$.  In this case, we say that the path~$p$ starts at the vertex $s(p) := s(e_{n})$ and ends at the vertex $r(p) := r(e_{1})$, we write $|p| = n$ for the length of $p$.  We consider the elements of $E^0$ to be paths of length $0$. We denote by $E^{*}$ the set of all paths in $E$.
A path  $p= e_{1} \cdots e_{n}$ of positive length is a \textit{cycle based at} the vertex $v$ if $s(p) = r(p) =v$ and the vertices $s(e_1), s(e_2), \hdots, s(e_n)$ are distinct. An \textit{infinite path} in $E$ is an infinite sequence $p = e_1 \cdots e_n\cdots$ of edges in $E$ such that $s(e_i) = r(e_{i+1})$ for all $i\geq 1$. In this case, we say that the infinite path $p$ ends at the vertex $r(p) := r(e_1)$. We denote by $E^{\infty}$ the set of all infinite paths in $E$.

An edge~$f$ is an \emph{entry} for a path $p = e_1 \dots e_n$ if $r(f) = r(e_i)$ but $f \ne e_i$ for some $1 \le i \le n$. A subset $H$ of $E^0$ is called \textit{hereditary} if $r(e)\in H$ implies $s(e)\in H$ for all $e\in E^1$. And $H$ is called \textit{saturated} if whenever $v$ is a regular vertex in $E^0$ with the property that $s(r^{-1}(v)) \subseteq H$, then $v\in H$.

\begin{defn}\label{LPAs}
For an arbitrary graph $E = (E^0,E^1,s,r)$
and any  field $K$, the \emph{Leavitt path algebra} $L_{K}(E)$ {\it of the graph}~$E$
\emph{with coefficients in}~$K$ is the $K$-algebra generated
by the union of the set $E^0$ and two disjoint copies of $E^1$, say $E^1$ and $\{e^*\mid e\in E^1\}$, satisfying the following relations for all $v, w\in E^0$ and $e, f\in E^1$:
\begin{itemize}
\item[(1)] $v w = \delta_{v, w} w$;
\item[(2)] $r(e) e = e = e s(e)$ and $s(e) e^* = e^* = e^*r(e)$;
\item[(3)] $e^* f = \delta_{e, f} s(e)$;
\item[(4)] $v= \sum_{e\in r^{-1}(v)}ee^*$ for any regular vertex $v$;
\end{itemize}
where $\delta$ is the Kronecker delta.
\end{defn}
It can be shown (\cite[Lemma 1.6]{ap:tlpaoag05}) that $L_K(E)$ is unital if and only if $E^0$ is finite; in this case the identity of $L_K(E)$ is $\sum_{v\in E^0}v$. The following theorem provides us with a complete characterization of simple Leavitt path algebras.

\begin{thm}[{cf. \cite[Theorem 3.11]{ap:tlpaoag05} and \cite[Theorem 2.9.1]{AAS}}]\label{simLPAs}
Let $E$ be an arbitrary graph and $K$ a field. Then $L_K(E)$ is simple if and only if the following three conditions are satisfied:
	
$(1)$ The only hereditary and saturated subset of~$E^0$ are~$\emptyset$ and~$E^0$;
	
$(2)$ Every cycle in~$E$ has an entry.
\end{thm}

We now describe the connection between Leavitt path algebras and Steinberg algebras. Let $E= (E^0, E^1, r, s)$ be a graph. Define
\begin{center}
$X_E: = \{p \in E^{*}\mid r(p)$ is either a source or an infinite receiver$\} \cup E^{\infty} $. 	
\end{center}
Let 
\begin{center}
$\mathcal{G}_E := \{(\alpha x, |\alpha| - |\beta|, \beta x)\mid \alpha, \beta \in E^*, x\in X_E, s(\alpha) = r(x) = s(\beta)\}.$
\end{center}
We view each $(x, k, y)\in \mathcal{G}_E$ as a morphism with range $x$ and source $y$. The formulas $(x, k, y) (y, l, z) = (x, k + l, z)$ and $(x, k, y)^{-1} = (y, -k, x)$ define composition and inverse maps on $\mathcal{G}_E$ making it a groupoid with $\mathcal{G}_E^{(0)} = \{(x, 0, x)\mid x\in X_E\}$ which we identify with the set $X_E$ by the map $(x, 0, x)\longmapsto x$.
We note that $r_{\mathcal{G}_E}$ and $s_{\mathcal{G}_E}: \mathcal{G}_E\longrightarrow \mathcal{G}_E^{(0)}$ are the range and source maps defined respectively by: $r_{\mathcal{G}_E}(x, k, y) = (x, 0, x)$ and $s_{\mathcal{G}_E}(x, 0, y) = (y, 0, y)$ for all $(x, k, y)\in \mathcal{G}_E$.  

We next describe a topology on $\mathcal{G}_E$. For $\alpha\in E^*$ and a finite subset $F\subseteq r^{-1}(s(\alpha))$, we define
$$Z(\alpha, \alpha) = \{(\alpha x, 0, \alpha x) \mid x\in X_E,\, s(\alpha) = r(x)\} \subseteq \mathcal{G}_E^{(0)}$$ and
$$Z(\alpha, \alpha, F) = Z(\alpha, \alpha)\setminus \bigcup_{e\in F}Z(\alpha e, \alpha e).$$
The sets $Z(\alpha, \alpha, F)$ constitute a basis of compact open sets for a locally compact Hausdorff topology on $\mathcal{G}_E^{(0)}$ (refer to \cite[Theorem 2.1]{w:tpsoadg} or \cite[Theorem 2.1]{r:tgatlpa} or \cite[Corollary 3.8]{adn:rcolpastawcctgca}).

For $\alpha, \beta\in E^*$ with $s(\alpha) = s(\beta)$, and a finite subset $F\subseteq r^{-1}(s(\alpha))$, we define
\[Z(\alpha, \beta) = \{(\alpha x, |\alpha| - |\beta|, \beta x)\mid x\in X_E, s(\alpha) = r(x) = s(\beta)\}\subseteq \mathcal{G}_E\] and \[Z(\alpha, \beta, F) = Z(\alpha, \beta)\setminus \bigcup_{e\in F}Z(\alpha e, \beta e).\]
The sets $Z(\alpha, \beta, F)$ constitute a basis of compact open bisections for a topology under which
$\mathcal{G}_E$ is a Hausdorff ample groupoid (refer to \cite[Subsection 2.3]{bcw:gaaoe} or \cite[Theorem 2.4]{r:tgatlpa}).  Thus we may form the Steinberg algebra $A_K(\mathcal{G}_E)$. By \cite[Example 3.2]{cs:eghmesa}, the map 
\[\pi_E: L_K(E)\longrightarrow A_K(\mathcal{G}_E),\]  
defined by $\pi_E(v) = 1_{Z(v, v)}$, $\pi_E(e) = 1_{Z(e, s(e))}$, and $\pi_E(e^*) = 1_{Z(s(e), e)}$ for all $v\in E^0$ and $e\in E^1$, extends to an algebra isomorphism. 

The center of the Leavitt path algebra of an arbitrary graph was completely described by Clark et al. in \cite{cbgm:utsamtdtcoalpa}. However, using the above isomorphism and Theorem~\ref{centofsimSteinAlg} directly, we obtain that simple Leavitt path algebras of arbitrary graphs with finitely many vertices are central, and nonunital simple Leavitt path
algebras have zero center, which extends Aranda Pino and Crow's result \cite[Theorem 4.2]{arancrow:nlpaatra} to the case of arbitrary graphs.

\begin{thm}
Let $K$ be a field and $E$ an arbitrary graph for which $L_K(E)$ is a simple Leavitt path algebra. Then the following holds:	

$(1)$ If $E^0$ is finite, then $Z(L_K(E)) = K\cdot 1_{L_K(E)}$.

$(2)$ if $E^0$ is infinite, then $Z(L_K(E)) = 0$.
\end{thm}
\begin{proof} We first have that $L_K(E)$ is isomorphic to $A_K(\mathcal{G}_E)$ as $K$-algebras (\cite[Example 3.2]{cs:eghmesa}), and $L_K(E)$ is unital if and only if $E^0$ is finite (by \cite[Lemma 1.6]{ap:tlpaoag05}).	From those notes and Theorem~\ref{centofsimSteinAlg}, we immediately obtain the theorem, thus finishing the proof.
\end{proof}

In \cite[Section 3]{am:slaaflpa} Abrams and Mesyan gave 
easily computable necessary and sufficient conditions to determine which Lie algebras of the form $[L_K(E), L_K(E)]$ are simple, when $E$ is a row-finite graph and $L_K(E)$ is simple. The next goal of this section is to extend Abrams and Mesyan's results to the case when $E$ is an arbitrary graph and $L_K(E)$ is simple.

Theorem~\ref{simpLieSteinalg} allows us to identify graphs $E$ for which simple Leavitt path algebra $L_K(E)$ yields a simple Lie algebra $[L_K(E), L_K(E)]$, which generalizes \cite[Corollaries 21 and 22]{am:slaaflpa}.

\begin{thm}\label{simpLieLPAalg}
Let $K$ be a field and $E$ an arbitrary graph for which $L_K(E)$ is a simple Leavitt path algebra. Then the following holds:
	
$(1)$ If $E^0$ is infinite, then $[L_K(E), L_K(E)]$ is a simple Lie $K$-algebra.
	
$(2)$ If $E^0$ is finite and $L_K(E)$ is nontrivial, then $[L_K(E), L_K(E)]$ is a simple Lie $K$-algebra if and only if $1_{L_K(E)}\notin [L_K(E), L_K(E)]$.
\end{thm}
\begin{proof}
We note that $L_K(E)$ is isomorphic to $A_K(\mathcal{G}_E)$ as $K$-algebras (\cite[Example 3.2]{cs:eghmesa}), and $L_K(E)$ is unital if and only if $E^0$ is finite (by \cite[Lemma 1.6]{ap:tlpaoag05}). Also, $A_K(\mathcal{G}_E)$ is unital if and only if $\mathcal{G}_E^{(0)}$ is compact in $\mathcal{G}_E$.
From those observations and Theorem~\ref{simpLieSteinalg}, we immediately obtain the theorem, finishing the proof.	
\end{proof}

In \cite[Theorem 14]{am:slaaflpa} Abrams and Mesyan gave (among other things) nice criteria for the element $1_{L_K(E)}$ to be written as sums of commutators of $L_K(E)$. We now recall this result. Before doing so, we need some notations. Let $K$ be a field and $E$ an arbitrary graph. We index the vertex set $E^0$ of $E$ by a set $I$, and write $E^0 = \{v_i\mid i\in I\}$.	Let $K^{(I)}$ denote the direct sum of copies of $K$ indexed by $I$. For each $i\in I$, let $\epsilon_i\in K^{(I)}$ denote the element with $1\in K$ as the $i$-th coordinate and zeros elsewhere.
	
\begin{defn}[{\cite[Definition 12]{am:slaaflpa}}]\label{B-space1}
Let $K$ be a field, let $E$ be an arbitrary graph, and write $E^0 = \{v_i\mid i\in I\}$. If $v_i$ is a regular vertex, for all $j\in I$ let $a_{ij}$ denote the number of edges $e\in E^1$ such that $r(e) = v_i$ and $s(e) = v_j$. In this situation, define
$B_i = (a_{ij})_{j\in I}- \epsilon_i \in  K^{(I)}$. On the other hand, let $B_i = (0)_{j\in I}\in K^{(I)}$, if $v_i$ is not a regular vertex.
\end{defn}

The following theorem describes various elements of a Leavitt path algebra $L_K(E)$ which may be written as sums of commutators.

\begin{thm}[{\cite[Theorem 14]{am:slaaflpa}}]\label{B-space2}
Let $K$ be a field, let $E$ be an arbitrary graph, and write $E^0 = \{v_i\mid i\in I\}$. For each $i\in I$ let $B_i$ denote the element of $K^{(I)}$ given in Definition~\ref{B-space1}, and let $\{k_i\mid i\in I\}$ be a set of scalars where $k_i =0$ for all but finitely many $i\in I$. Then
\begin{center}
$\sum_{i\in I}k_iv_i \in [L_K(E), L_K(E)]$ if and only if $(k_i)_{i\in I}\in \rm{Span}_K\{B_i\mid i\in I\}\subseteq K^{(I)}.$
\end{center} In particular, if $E^0$ is finite, then
\begin{center}
$1_{L_K(E)} \in [L_K(E), L_K(E)]$ if and only if $(1, \hdots , 1)\in \rm{Span}_K\{B_i\mid i\in I\}\subseteq K^{(I)}.$\end{center}
\end{thm}

Now combining Theorems~\ref{simpLieLPAalg} (1) and~\ref{B-space2}, we obtain a easily computable criterion for Lie algebras of the form $[L_K(E), L_K(E)]$ are simple, when $E$ is an arbitrary graph and $L_K(E)$ is simple, which generalizes Abrams and Mesyan's result \cite[Theorem 23]{am:slaaflpa}.

\begin{thm}\label{B-space3}
Let $K$ be a field, let $E$ be an arbitrary graph with finitely many vertices for which $L_K(E)$ is a nontrivial simple Leavitt path algebra. Write $E^0 = \{v_1, \hdots , v_m\}$. For each $1\le i\le m$ let $B_i$ denote the element of $K^{(I)}$ given in Definition~\ref{B-space1}. Then the Lie $K$-algebra $[L_K(E), L_K(E)]$ is simple if and only if $(1, \hdots , 1)\notin \rm{Span}_K\{B_1, \hdots , B_m\}.$	
\end{thm}


We close this section with the following example. 

\begin{exas}\label{LPAs are not SteinBerg Alg}
Let $R_{\mathbb{N}} = (R_{\mathbb{N}}^0, R_{\mathbb{N}}^1, r, s)$ be the graph with $R_{\mathbb{N}}^0 = \{v_1\}$, $R_{\mathbb{N}}^1 =\{e_n\mid n\in \mathbb{N}\}$ and $r(e_n) = s(e_n) = v_1$ for all $n\in\mathbb{N}$. Then, the Lie $K$-algebra $[L_{K}(R_{\mathbb{N}}), L_{K}(R_{\mathbb{N}})]$ is simple for all field $K$. 
\end{exas}
\begin{proof}
It is obvious that $R_{\mathbb{N}}^0$ has the trivial hereditary and saturated subsets, and every cycle in $R_{\mathbb{N}}$ has an exit. Therefore, by Theorem~\ref{simLPAs}, $L_K(R_{\mathbb{N}})$ is simple. We have $B_1 = 0 \in K$, and hence $1\notin \rm{Span}_K\{B_1\} = 0$. Then, by Theorem~\ref{B-space3}, $[L_{K}(R_{\mathbb{N}}), L_{K}(R_{\mathbb{N}})]$ is a simple Lie $K$-algebra, finishing the proof.	
\end{proof}

\section{Application: Kumjian-Pask algebras}
In this section, based on Section 2, we describe the center of a Kumjian-Pask algebra $KP_K(\Lambda)$ (Theorem~\ref{cenofsimKPalg}), and give criteria (Theorem~\ref{simpLieKPalg}) to determine which Lie algebras of the form $[KP_K(\Lambda), KP_K(\Lambda)]$ are simple, when $\Lambda$ is a row-finite $k$-graph without sources and $KP_K(\Lambda)$ is simple. 

For a positive integer $k$, we view the additive semigroup $\mathbb{N}^k$ as a category with one object. Following Kumjian and Pask \cite{kp:hrgca}, a \textit{graph of rank} $k$ or $k$-graph is a countable category $\Lambda = (\Lambda^0, \Lambda, s, r)$ together with a functor 
$d: \Lambda \longrightarrow \mathbb{N}^k$, called \textit{the degree map}, satisfying the following factorisation property: if $\lambda\in \Lambda$ and $d(\lambda) = m + n$ for some $m, n\in \mathbb{N}^k$, then there are unique $\mu, \nu\in \Lambda$ such that $d(\mu) = m$, $d(\nu) = n$ and $\lambda = \mu\nu$. The functor $d$ is called the \textit{degree functor} and $d(\lambda)$ is called the \textit{degree} of $\lambda$. Using the unique factorization property, we identify the set of objects
$\Lambda^0$ with the set of morphisms of degree $0$, that means, $\Lambda^0 = \{\lambda \in \Lambda\mid d(\lambda) = 0\}$. Then, for $n\in \mathbb{N}^k$, we write $\Lambda^n := d^{-1}(n)$, and call the elements $\lambda$ of $\Lambda^n$ \textit{paths of degree n from $s(\lambda)$ to $r(\lambda)$}. For $v\in \Lambda^0$ we write $v\Lambda^n$ or $v\Lambda$ for the sets of paths with range $v$ and $\Lambda^n v$ or $\Lambda v$  for paths with source $v$. We say that $\Lambda$ is \textit{row-finite} if $v\Lambda^n$ is finite for every $v \in \Lambda^0$ and $n\in \mathbb{N}^k$; we say that $\Lambda$ \textit{has no sources} if  $v\Lambda^n$ is nonempty for every $v \in \Lambda^0$ and $n\in \mathbb{N}^k$.

An important example is the $k$-graph $\Omega_k$ defined as a set by $\Omega_k =\{(m, n)\in \mathbb{N}^k\times \mathbb{N}^k\mid m\le n\}$ with $d(m, n) = m-n$, $\Omega_k^0 = \mathbb{N}^k$, $r(m, n) = m$, $s(m,n) = n$ and $(m, n) (n, p) = (m, p)$.

In \cite{kp:hrgca} Kumjian and Pask introduced the $C^*$-algebra associated to a higher rank graph as higher-rank generalization of graph $C^*$-algebras. This algebra has generated a great deal of interest among operator algebraists and has broadened the class
of $C^*$-algebras that can be realized as graph algebras.
In \cite{acar:kpaohrg} Aranda Pino et al. introduced the Kumjian-Pask algebra of a higher rank graph as the algebraic version of higher-rank graph $C^*$-algebras.

\begin{defn}
Let $\Lambda$ be a row-finite $k$-graph without sources and $K$ a field. The \textit{Kumjian-Pask $K$-algebra} $KP_K(\Lambda)$ of $\Lambda$ is the $K$-algebra generated by $\Lambda \cup \Lambda^*$	subject to the relations:

\begin{itemize}
\item[(KP1)] $v w = \delta_{v, w} w$ for all $v, w\in \Lambda^0$;

\item[(KP2)] $\lambda\mu = \lambda\circ\mu$ and $\mu^*\lambda^* = (\lambda\circ\mu)^*$ for all $\lambda,\, \mu\in \Lambda$ with $r(\mu) = s(\lambda)$;

\item[(KP3)] $\lambda^* \mu = \delta_{\lambda, \mu} r(\lambda)$ for all $\lambda,\, \mu\in \Lambda$ with $d(\mu) = d(\lambda)$;

\item[(KP4)] $v= \sum_{\lambda\in v\Lambda^n}\lambda\lambda^*$ for all $v\in \Lambda^0$ and all $n\in \mathbb{N}^k\setminus\{0\}$;
\end{itemize}
where $\delta$ is the Kronecker delta.
\end{defn}

The class of Kumjian-Pask algebras over a field is strictly larger than the class of Leavitt path algebras over that field (see \cite[Example 7.1]{acar:kpaohrg}).
The Kumjian-Pask $K$-algebra $KP_K(\Lambda)$ is unital if and only if $\Lambda^0$ is finite; and in this case the identity of $KP_K(\Lambda)$ is $\sum_{v\in \Lambda^0}v$ (see, e.g., \cite[Lemma 4.6]{ba:coaathrg}).

Let $\Lambda$ be a row-finite $k$-graph without sources. Following \cite[Section 2]{kp:hrgca}, an \textit{infinite path} in $\Lambda$ is a degree-preserving functor $x: \Omega_k\longrightarrow \Lambda$. Denote the set of all infinite paths by $\Lambda^{\infty}$. We write $x(m)$ for the vertex $x(m, m)$. Then the range of an infinite path $x$ is the vertex
$r(x) := x(0)$.

For $\lambda\in \Lambda$, set $Z(\lambda) = \{x\in \Lambda^{\infty}\mid x(0, d(\lambda))= \lambda\}$. Then $\{Z(\lambda)\mid \lambda\in \Lambda\}$ is a basis
for a topology, and we equip $\Lambda^{\infty}$ with this topology. Then $\Lambda^{\infty}$ is a totally disconnected,
locally compact Hausdorff space, and each $Z(\lambda)$ is compact and open. For $p\in \mathbb{N}^k$ define $\sigma^p: \Lambda^{\infty} \longrightarrow \Lambda^{\infty}$ by $\sigma^p(x)(m,n) = x(m+p, n+p)$.

By \cite[Definition 4.3]{kp:hrgca}, a $k$-graph is \textit{aperiodic} if for every $v\in \Lambda^0$ there exists $x\in Z(v)$ such that $\sigma^m(x)\neq \sigma^n(x)$ for all distinct $m, n\in \mathbb{N}^k$. Following \cite[Definition 4.1]{kp:hrgca}, a $k$-graph $\Lambda$ is called \textit{cofinal} if for every infinite path $x$ and every vertex $v$, there exists $m\in \mathbb{N}^k$ such that $v\Lambda x(m) \neq \varnothing$.

A description of the row-finite $k$-graph $\Lambda$ without sources and a fields $K$ for which $KP_K(\Lambda)$ is simple given in \cite[Theorem 6.1]{acar:kpaohrg} (see also \cite[Theorem 9.1]{cp:kpaofahrg}).

\begin{thm}[{\cite[Theorem 6.1]{acar:kpaohrg} and \cite[Theorem 9.1]{cp:kpaofahrg}}]\label{simKPalg}
Let $K$ be a filed and  $\Lambda$ a row-finite $k$-graph without sources. Then $KP_K(\Lambda)$ is simple if and only if $\Lambda$ is aperiodic and cofinal.	
\end{thm}

As in the theory of Leavitt path algebras, one can model Kumjian-Pask algebras as Steinberg algebras via the
infinite-path groupoid of the $k$-graph (see \cite[Proposition 5.4]{cp:kpaofahrg}). For $k$-graph $\Lambda$, 
\[\mathcal{G}_{\Lambda} = \{(x, m-n, y)\in\Lambda^{\infty}\times \mathbb{Z}^k\times \Lambda^{\infty}\mid \sigma^m(x) =\sigma^n(y) \}.\] Then $\mathcal{G}_{\Lambda}$  is a groupoid with composition and inverse given by
\begin{center}
$(x, l, y) (y, m, z) = (x, l+m, z)$ and $(x, l, y)^{-1} = (y, -l, x)$.	
\end{center}
For $\mu, \nu\in \Lambda$ with $s(\mu) = s(\nu)$ set \[Z(\mu, \nu) = \{(\mu z, d(\mu)-d(\nu), \nu z)\mid z\in Z(s(\mu))\}.\] Then $\{Z(\mu, \nu)\mid \mu, \nu\in \Lambda, \, s(\mu) = s(\nu)\}$
is a basis for a topology on $\mathcal{G}_{\Lambda}$. Then $\mathcal{G}_{\Lambda}$ is an ample Hausdorff groupoid (see \cite[Proposition 2.8]{kp:hrgca}). The unit space $\mathcal{G}_{\Lambda}^{(0)}$ is $\{(x, 0, x)\mid x\in \Lambda^{\infty}\}$, which we identify with $\Lambda^{\infty}$; the identification takes $Z(\mu, \mu)$ to $Z(\mu)$. 

Let $K$ be a filed. The Kumjian--Pask algebra $KP_K(\Lambda)$ is canonically isomorphic to the Steinberg algebra $A_K(\mathcal{G}_{\Lambda})$. This was proved in  
\cite[Proposition 4.3]{cfst:aggolpa} when $K = \mathbb{C}$, and for a finitely aligned $k$-graph in \cite[Proposition 5.4]{cp:kpaofahrg}. (A row-finite $k$-graph
with no sources is finitely aligned.)

The center of the  Kumjian--Pask algebra of a row-finite $k$-graph was completely described by Brown and an Huef in \cite[Theorem 4.7]{ba:coaathrg}. However, using Theorem~\ref{centofsimSteinAlg} directly, we immediately obtain the following:

\begin{thm}[{cf. \cite[Theorem 4.7]{ba:coaathrg}}]\label{cenofsimKPalg}
Let $K$ be a filed and  $\Lambda$ a row-finite $k$-graph without sources for which $KP_K(\Lambda)$ is simple. Then the following holds:

$(1)$ If $\Lambda^0$ is finite, then $Z(KP_K(\Lambda)) = K\cdot 1_{KP_K(\Lambda)}.$

$(1)$ If $\Lambda^0$ is infinite, then $Z(KP_K(\Lambda)) = 0.$	
\end{thm}	
\begin{proof}
We note that $KP_K(\Lambda)\cong A_K(\mathcal{G}_{\Lambda})$ (by \cite[Proposition 5.4]{cp:kpaofahrg}) and $KP_K(\Lambda)$ is unital if and only $\Lambda^0$ is finite (by \cite[Lemma 4.6]{ba:coaathrg}). Then, by Theorem~\ref{centofsimSteinAlg}, we obtain the theorem, finishing the proof.	
\end{proof}	 

Theorems~\ref{simpLieSteinalg}, \ref{simKPalg} and \ref{cenofsimKPalg} allow us to identify $k$-graphs $\Lambda$ without sources for which the simple Kumjian--Pask algebra $KP_K(\Lambda)$ yields a simple Lie algebra $[KP_K(\Lambda), KP_K(\Lambda)]$.

\begin{thm}\label{simpLieKPalg}
Let $K$ be a field and $\Lambda$ a row-finite $k$-graph without sources. Then the following holds:
	
$(1)$ If $\Lambda$ is aperiodic and cofinal, and $\Lambda^0$ is infinite, then  $[KP_K(\Lambda), KP_K(\Lambda)]$ is a simple Lie $K$-algebra.
	
$(2)$ If $\Lambda$ is aperiodic and cofinal, $\Lambda^0$ is finite, and $KP_K(\Lambda)$ is not isomorphic to $K$, then  $[KP_K(\Lambda), KP_K(\Lambda)]$ is a simple Lie $K$-algebra if and only if $1_{KP_K(\Lambda)}\notin  [KP_K(\Lambda), KP_K(\Lambda)]$.
\end{thm}

\section{Application: Exel-Pardo algebras}
In this section, based on Section 2, we calculus the center of simple Exel-Pardo algebras $L_K(G, E)$ (Theorem~\ref{cenofsimEPalg}), and give criteria for Lie algebras of the form $[L_K(G, E), L_K(G, E)]$ are simple (Theorem~\ref{simLieEPalg}), when $L_K(G, E)$ is simple.

Let $E= (E^0, E^1, r, s)$ be a graph. Following \cite{exelpar:ssgautokanca},  by an \textit{automorphism} of $E$
we mean a bijective map $\sigma: E^0\sqcup E^1\longrightarrow E^0\sqcup E^1$ such that $\sigma(E^i) = E^i$ $(i= 0, 1)$, $r\circ\sigma = \sigma\circ r$ and $s\circ\sigma = \sigma\circ s$.
By an \textit{action} of a group $G$ on $E$ we shall mean a group homomorphism $g\longmapsto \sigma_g$  from $G$ to the group of all automorphisms of $E$. We often write $g \cdot e$ instead of $\sigma_g(e)$. The unit in a group $G$ is denoted $e_G$.

Let $X$ be a set, and let $\sigma$ be an action of a group $G$ on $X$. A map $\phi: G\times X\longrightarrow X$ is a \textit{one-cocycle} for $\sigma$ if \[\phi(gh, x) = \phi(g, \sigma_g(x))\phi(h, x)\] for all $g, h\in G$ and all $x\in X$, see \cite{exelpar:ssgautokanca}.

\begin{nota}\label{nota}
The quadruple $(G, E, \sigma, \phi)$, sometimes written as a
triple $(G, E, \phi)$ or a pair $(G, E)$, will denote a countable discrete group $G$, a row-finite graph $E$ with no sources, an action $\sigma$ on $E$, a one-cocycle $\phi: G\times E^1\longrightarrow E^1$  for the restriction of $\sigma$ to $E^1$ such that for all $g\in G$, $e\in E^1$, $v\in E^0$ \[\phi(g, e)  \cdot v= g\cdot v.\]
\end{nota}	

In \cite{exelpar:ssgautokanca} Exel and Pardo introduced $C^*$-algebras $\mathcal{O}_{G, E}$ giving a unified treatment of two classes of $C^*$-algebras which have attracted significant recent attention,  namely Katsura $C^*$-algebras and Nekrashevych's self-similar group $C^*$-algebras. This suggests that algebraic analogues of Exel-Pardo $C^*$-algebras may be a source of interesting new examples. Partial results have already been established by Clark, Exel and Pardo in \cite{cep:agutatgisosa} who introduced $\mathcal{O}^{alg}_{G, E}(R)$ for finite graphs $E$. In \cite{hpss:aaaoepa} Hazrat et al. introduced an algebraic analogue $L_K(G, E)$ of Exel-Pardo $C^*$-algebras for row-finite graphs $E$. 

\begin{defn}[{\cite[Theorem 1.6]{hpss:aaaoepa}}]
Let $(G, E, \phi)$ be as in Notation~\ref{nota} and $K$ a field (with the trivial involution). The \textit{Exel-Pardo $K$-algebra}
$L_K(G, E)$ is the $*$-algebra over $K$ generated by \[\{P_{v, f}\mid v\in E^0, \, f\in G\} \cup \{S_{e, g}\mid e\in E^1, \, g\in G\}\] subject to the relations:
\begin{itemize}
\item[(a)] $\{P_{v, e_G}\mid v\in E^0\} \cup \{S_{e, e_G}\mid e\in E^1\}$  satisfies the analogous to (1)-(4) relations in Definition~\ref{LPAs};

\item[(b)] $(P_{v, f})^* = P_{f^{-1}\cdot v, f^{-1}}$;

\item[(c)] $P_{v, f} P_{w, h} = \delta_{v, f\cdot w}P_{v, fh}$;

\item[(d)] $P_{v, f} S_{e, g} = \delta_{v, r(f\cdot e)}S_{f \cdot e,\phi(f, e)g}$;

\item[(e)] $S_{e, g} P_{v, f} = \delta_{g\cdot v, s(e)}S_{e, gf}$;
\end{itemize}
where $\delta$ is the Kronecker delta.
\end{defn}	
We note that the generators $P_{v, f}$ and $S_{e, g}$ of $L_K(G, E)$ are all nonzero; see \cite[Proposition 3.7]{hpss:aaaoepa}. Furthermore, we have the following:

\begin{lem}\label{unitalKPalg}
Let $(G, E, \phi)$ be as in Notation~\ref{nota} and $K$ a field. Then $L_K(G, E)$ is unital if and only if $E^0$ is finite; in this case $1_{L_K(G, E)} = \sum_{v\in E^0}P_{v, e_G}$.	
\end{lem}
\begin{proof}
If $E^0$ is finite, then $\sum_{v\in E^0}P_{v, e_G}$ is clearly an identity for $L_K(G, E)$. Conversely, assume that $L_K(G, E)$ has an identity $1_{L_K(G, E)}$. Let \[X = \{P_{v, f}, S_{e, g}, (P_{v, f})^*, (S_{e, g})^*\mid v\in E^0, e\in E^1, f, g\in G\}\] and $Y: = \omega(X)$ the set of all finite words in the alphabet $X$. Then, every element of $L_K(G, E)$ can be written in the form $\sum_{y\in Y}k_yy$ in which all but
finitely many coefficients $k_y\in K$ are zero. Write $1_{L_K(G, E)} = k_1 y_1 + \cdots + k_n y_n$, where $k_i\in K\setminus \{0\}$ and $y_i\in Y$. For each $1\le i\le n$, let $v_i$ and $w_i\in E^0$ such that $P_{v_i, e_G} y_i = y_i = y_i P_{w_i, e_G}$. Let $W = \{v_i, w_i\mid \le i\le n\}\subseteq E^0$ and $\epsilon = \sum_{w\in W}P_{w, e_G}$. We then have $\epsilon = \epsilon \cdot 1_{L_K(G, E)} = \epsilon (k_1 y_1 + \cdots + k_n y_n) = k_1 y_1 + \cdots + k_n y_n = 1_{L_K(G, E)}$. If $E^0$ is infinite, then there exists $v\in E^0\setminus W$. We then have $P_{v, e_G} = P_{v, e_G}\cdot 1_{L_K{G, E}} = P_{v, e_G}\cdot \epsilon = 0$. On the other hand, by \cite[Proposition 3.7]{hpss:aaaoepa}, $P_{v, e_G}\neq 0$, a contradiction. Therefore $E^0$ is finite, thus finishing the proof.
\end{proof}

Recall from \cite[Definition 4.1]{exelpar:ssgautokanca} that given a triple $(G, E, \phi)$ as in Notation~\ref{nota}, we define the inverse semigroup $\mathcal{S}_{G, E}$ as follows: 
\[\mathcal{S}_{G, E} = \{(\alpha, g, \beta)\mid \alpha, \beta\in E^*, g\in G, s(\alpha)= g\cdot s(\beta)\}\cup \{0\}.\] The proof of \cite[Proposition 4.3]{exelpar:ssgautokanca} shows that under the multiplication
\begin{align*}(\alpha, g,\beta)(\gamma, h, \delta)= \begin{cases} (\alpha g\cdot \epsilon, \phi(g, \epsilon)h, \delta) &\textnormal{if }  \gamma = \beta\epsilon, \\ (\alpha, g\phi(h^{-1}, \epsilon)^{-1}, \delta (h^{-1})\cdot \epsilon) & \textnormal{if } \beta = \gamma\epsilon, \\  0 &\textnormal{otherwise}.\end{cases}\end{align*}	
$\mathcal{S}_{G, E}$ is an inverse semigroup, in which $(\alpha, g, \beta)^* = (\beta, g^{-1}, \alpha)$ and where
$0$ acts as a zero in $\mathcal{S}_{G, E}$. Then, we can construct the groupoid of germs of the action of $\mathcal{S}_{G, E}$ on the compact space $\widehat{\mathcal{E}}$ of characters of the semilattice $\mathcal{E}$ of idempotents of $\mathcal{S}_{G, E}$. In our concrete case, $\widehat{\mathcal{E}}$ turns out to be homeomorphic to the compact space $E^{\infty}$ of infinite paths on $E$; the action of $(\alpha, g, \beta)\in \mathcal{S}_{G, E}$ on $\eta = \beta\widehat{\eta}$ is given by the rule $(\alpha, g, \beta)\cdot \eta = \alpha(g\widehat{\eta})$. Thus, the groupoid of germs is \[\mathcal{G}_{(G, E)} = \{[\alpha, g, \beta; \eta]\mid \eta = \beta\widehat{\eta}\},\] where $[s; \eta] = [t; \mu]$ if and only if $\eta = \mu$ and there exists $0 \neq e^2 = e \in \mathcal{S}_{G, E}$ such that $e\cdot \eta = \eta$ and $se = te$. The unit space \[\mathcal{G}_{(G, E)}^{(0)} = \{[\alpha, e_G, \alpha; \eta]\mid \eta = \alpha\widehat{\eta}\}\] is identified with the  infinite path space $E^{\infty}$, via the homeomorphism
$[\alpha, e_G, \alpha; \eta]\longmapsto \eta$. Under this identification, the range and source maps on $\mathcal{G}_{(G, E)}$ are:
\begin{center}
$s([\alpha, g, \beta; \beta\widehat{\eta}]) = \beta\widehat{\eta}$ and $r([\alpha, g, \beta; \beta\widehat{\eta}]) = \alpha(g\widehat{\eta})$.	
\end{center}

A basis for the topology on $\mathcal{S}_{G, E}$ is given by compact open bisections of the form \[\Theta(\alpha, g, \beta; Z(\gamma)) := \{[\alpha, g, \beta; \xi]\in \mathcal{G}_{(G, E)}\mid \xi \in Z(\gamma)\}\] where $\gamma\in E^*$ and $Z(\gamma) := \{\gamma\widehat{\eta}\mid \widehat{\eta}\in E^{\infty}\}$. Thus $\mathcal{G}_{(G, E)}$ is locally compact and ample with a Hausdorff unit space, see \cite[Proposition 4.14]{exel:isacca}.

Due to work in \cite{exelparstar:caossgoag} it is known when the groupoid $\mathcal{G}_{(G, E)}$ is Hausdorff. We recall the relevant terminology. A path $\alpha\in E^*$ is strongly fixed by $g\in G$ if $g\cdot \alpha = \alpha$ and $\phi(g, \alpha) = e_G$.
In addition if no prefix of $\alpha$ is strongly fixed by $g$, we say that $\alpha$ is a \textit{minimal strongly fixed} path for $g$ (see \cite[Definition 5.2]{exelpar:ssgautokanca}). By \cite[Theorem 4.2]{exelparstar:caossgoag}, $\mathcal{G}_{(G, E)}$ is Hausdorff if and only if for every $g\in G$, and every $v\in E^0$, there are at most finitely many minimal strongly fixed paths
for $g$ with range $v$. \cite[Theorem B]{hpss:aaaoepa} showed that every Exel-Pardo algebra is a Steinberg algebra.

\begin{thm}[{\cite[Theorem B and Corollary 3.13]{hpss:aaaoepa}}]\label{EPalg=Steinalg}
Let $(G, E, \phi)$ be as in Notation~\ref{nota} and $K$ a field (with the trivial involution). Suppose that for every $g\in G$, and every $v\in E^0$, there are at most finitely many minimal strongly fixed paths for $g$ with range $v$. Then \[L_K(G, E)\cong A_K(\mathcal{G}_{G, E}).\]	
\end{thm}

Since the focus of this section is simple Exel-Pardo algebras, and a criterion for simpleness of Exel-Pardo algebras was given in \cite[Proposition 3.14]{hpss:aaaoepa} (see also \cite[Theorem 4.5]{exelparstar:caossgoag}),we would like to recall this criterion. First we give the relavent definitions from \cite{exelpar:ssgautokanca}. We say that $E$ is \textit{weakly $G$-transitive} if, given any infinite path $\alpha$, and any vertex $v\in E^0$,  there is some vertex $w$ along $\alpha$ such that there exists a vertex $u$ with
$u = gv$ for some $g\in G$ and there is a path from $w$ to $u$.
A \textit{$G$-circuit} is a pair $(g, \gamma)$, where $g\in G$ and $\gamma\in E^*$ is a path of positive length such that 
$s(\gamma) = g r(\gamma)$. Given a $G$-circuit $(g, \gamma)$
such that $\gamma = \gamma_1 \cdots \gamma_n$ in $E^*$ and each 
$\gamma_i$ in $E^1$, we say that $\gamma$ has no \textit{entry}
if $r^{-1}(s(\gamma_i)) = \{\gamma_{i+1}\}$ for all $i = 1, \hdots , n-1$, and $r^{-1}(s(\gamma_n)) = \{g\gamma_{1}\}$. Given $g\in G$ and $v\in E^0$, we shall say that $g$ is \textit{slack} at $v$ if there is a non-negative integer $n$ such that all paths
$\gamma\in E^*$ with $r(\gamma) = v$ and $|\gamma|\ge n$, are
strongly fixed by $g$. 

\begin{thm}[{\cite[Proposition 3.13]{hpss:aaaoepa} and \cite[Theorem 4.5]{exelparstar:caossgoag}}]\label{simEPalg}
Let $(G, E, \phi)$ be as in Notation~\ref{nota} and $K$ a field (with the trivial involution). Suppose that for every $g\in G$, and every $v\in E^0$, there are at most finitely many minimal strongly fixed paths for $g$ with range $v$. Then $L_K(G, E)$ is simple if and only if the following condition are satisfied:

$(1)$ the graph $E$ is weakly $G$-transitive;

$(2)$ every $G$-circuit has an entry; 

$(3)$ for every vertex $v$, and every $g\in G$ fixing $Z(v)$ pointwise, $g$ is slack at $v$.
\end{thm}	

We are now in position to provide the main results of this section. The following theorem completely describes the center of simple Exel-Pardo algebras.

\begin{thm}\label{cenofsimEPalg}
Let $(G, E, \phi)$ be as in Notation~\ref{nota} and $K$ a field (with the trivial involution). Suppose that for every $g\in G$, and every $v\in E^0$, there are at most finitely many minimal strongly fixed paths for $g$ with range $v$. The following holds:

$(1)$ If $E^0$ is finite and $L_K(G,E)$ is simple, then $Z(L_K(G,E)) = K\cdot 1_{L_K(G,E)}$.

$(2)$ If $E^0$ is infinite and $L_K(G,E)$ is simple, then $Z(L_K(G,E)) = 0$.
\end{thm}
\begin{proof}
The theorem follows from Lemma~\ref{unitalKPalg} and Theorems~\ref{centofsimSteinAlg} and \ref{EPalg=Steinalg}.
\end{proof}	

The theorem allows us to give necessary and sufficient conditions to determine which Lie algebras of the form $[L_K(G, E), L_K(G, E)]$ are simple, when $(G, E)$ is as in Notation~\ref{nota} and $L_K(G, E)$ is simple.

\begin{thm}\label{simLieEPalg}
Let $(G, E, \phi)$ be as in Notation~\ref{nota} and $K$ a field (with the trivial involution). Suppose that for every $g\in G$, and every $v\in E^0$, there are at most finitely many minimal strongly fixed paths for $g$ with range $v$. The following holds:
	
$(1)$ If $E^0$ is infinite and $L_K(G,E)$ is simple, then $[L_K(G, E), L_K(G, E)]$ is a simple Lie $K$-algebra.
	
$(2)$ If $E^0$ is finite and $L_K(G,E)$ is a nontrivial simple Exel-Pardo algebra, then the Lie $K$-algebra $[L_K(G, E), L_K(G, E)]$ is simple if and only if $1_{L_K(G, E)}\notin [L_K(G, E), L_K(G, E)]$.
\end{thm}
\begin{proof}
The theorem	follows from Lemma~\ref{unitalKPalg} Theorems~\ref{simpLieSteinalg} and \ref{EPalg=Steinalg}.
\end{proof}	


%
%
%
%

\vskip 0.5 cm \vskip 0.5cm {

\end{document}
\begin{thebibliography}{99}



\bibitem{AAS} G. Abrams, P. Ara, and M. Siles Molina, \emph{Leavitt path algebras}.   Lecture Notes in Mathematics series, Vol. 2191, Springer-Verlag Inc., 2017.

\bibitem{ap:tlpaoag05} G.~Abrams and G.~Aranda Pino, The Leavitt path algebra of a graph, \emph{J. Algebra} \textbf{293} (2005), 319--334.




\bibitem{adn:rcolpastawcctgca}
G. Abrams, M. Dokuchaev and T. G. Nam, Realizing corners of Leavitt path algebras as Steinberg algebras, with corresponding connections to graph $C^*$-algebras, arXiv:1909.03964v1.


\bibitem{am:slaaflpa} G. Abrams, Z. Mesyan, Simple Lie algebras arising from Leavitt path algebras, \emph{J. Pure Appl. Algebra} {\bf 216} (2012), 2302-2313.










\bibitem{amp:nktfga} P.~Ara, M.\,A.~Moreno, and E.~Pardo, Nonstable K-theory for graph algebras, \emph{Algebr.\ Represent.\ Theory} \textbf{10} (2007), 165--224.


\bibitem{acar:kpaohrg} G. Aranda Pino, J. Clark, A. an Huef and I. Raeburn, Kumjian--Pask algebras of higher-rank graphs,  \textit{ Trans. Amer. Math. Soc.} \textbf{365} (2013), 3613--3641.


\bibitem{arancrow:nlpaatra} G. Aranda Pino and K. Crow, The center of a Leavitt path algebra,  \textit{Rev. Mat. Iberoam.} \textbf{27} (2011), 621--644.










 
\bibitem{ba:coaathrg} J. Brown and A. an Huef, Centers of algebras associated to higher-rank graphs,  \textit{Rev. Mat. Iberoam.} \textbf{30} (2014), 1387--1396. 
 
\bibitem{bcfs:soaateg}
J. Brown, L.O. Clark, C. Farthing and A. Sims, Simplicity of algebras associated to \'{e}tale groupoids, \emph{Semigroup Forum} \textbf{88} (2014), 433--452. 

 
\bibitem{bcw:gaaoe} N. Brownlowe, T. M. Carlsen, M. F. Whittaker, Graph algebras and orbit equivalence, \textit{Ergodic Theory Dynam. Systems} \textbf{37} (2017),  389--417. 




\bibitem{ce:utfsa} L. O. Clark, C. Edie-Michell, Uniqueness theorems for Steinberg algebras, \textit{Algebr. Represent. Theory} \textbf{18} (2015), 907-916.

\bibitem{cep:agutatgisosa}
L. O. Clark, R. Exel, and E. Pardo, A generalized uniqueness theorem and the graded ideal structure of Steinberg algebras, \textit{Forum Math.} 303 (2018), 533-552.

\bibitem{cfst:aggolpa}
L.O. Clark, C. Farthing, A. Sims, and M. Tomforde, A groupoid generalisation of Leavitt path algebras \emph{Semigroup Forum} 
\textbf{89} (2014), 501--517. 

\bibitem{cbgm:utsamtdtcoalpa}
L. O. Clark, D. Mart\'{\i}n Barquero, C. Mart\'{\i}n Gonz\'{a}lez, M. Siles Molina, Using the Steinberg algebra model to determine the center of any Leavitt path algebra,  \textit{Israel J. Math.} \textbf{230} (2019), 23--44. 

\bibitem{cp:kpaofahrg}
L.O. Clark, Y. E. P. Pangalela, Kumjian--Pask algebras of finitely aligned higher-rank graphs , \textit{J. Algebra} \textbf{482} (2017), 364--397.


\bibitem{cs:eghmesa}
L.O. Clark, A. Sims, Equivalent groupoids have Morita equivalent Steinberg algebras, \textit{J. Pure Appl. Algebra} \textbf{219} (2015), 2062--2075.







  
 

\bibitem{exel:isacca} R. Exel, Inverse semigroups and combinatorial $C^*$-algebras,  \textit{Bull. Braz. Math. Soc.} \textbf{39} (2008), 191-313. 

\bibitem{exelpar:ssgautokanca} R. Exel, E. Pardo,  Self-similar graphs, a unified treatment of Katsura and Nekrashevych $C^*$-algebras,  \textit{Adv. Math.} \textbf{306} (2017), 1046-1129. 

\bibitem{exelparstar:caossgoag} R. Exel, E. Pardo, and C. Starling,  $C^*$-algebras of self-similar graphs over arbitrary graphs, arXiv:1807.01686.


\bibitem{hpss:aaaoepa} R. Hazrat, D. Pask, A. Steranowski, and A. Sims, An algebraic analogue of Exel-Pardo $C^*$-algebras, arXiv:1912.1211v1.

\bibitem{Her:tirt} I. N. Herstein, Topics in ring theory, in: Mathematics Lecture Notes, University of Chicago, 1965.








\bibitem{jj:lrodoar}
C.R. Jordan, D.A. Jordan, Lie rings of derivations of associative rings, \textit{J. Lond. Math. Soc.} \textbf{2} (1978) 33--41.

\bibitem{kp:hrgca}
A. Kumjian, D. Pask, Higher rank graph $C^*$-algebras, \textit{New York J. Math.} \textbf{6} (2000), 1--20. 









\bibitem{mes:cr} Z. Mesyan, Commutator rings, \emph{Bull. Aust. Mat. Soc.} \textbf{74} (2006), 279--288.






\bibitem{p:gisatoa} A. L. T. Paterson,  \textit{Groupoids, Inverse Semigroups, and their Operator Algebras}, vol. 170 of Progress in Mathematics. Springer, 1999.
 
\bibitem{r:agatca} I. Raeburn, A groupoid approach to $C^*$-algebras, Lecture Notes in Mathematics, vol. 793, Springer, Berlin, 1980 


\bibitem{r:egatq}  P. Resende, \'{E}tale groupoids and their quantales,  \textit{Adv. Math.} \textbf{208} (2007), 147--209.
 

\bibitem{r:tgatlpa}  S. W. Rigby, The groupoid approach to Leavitt path algebras,   in ``Leavitt Path Algebras and Classical $K$-Theory",  A. A. Ambily, R. Hazrat, and B. Sury eds.  Indian Statistical Institute Series, Springer (2019), pp. 23--71.






\bibitem{stein:agatdisa}
B. Steinberg, A groupoid approach to discrete inverse semigroup algebras, \emph{Adv. Math.} \textbf{223} (2010), 689--727.






\bibitem{w:tpsoadg}
S. B. G. Webster, The path space of a directed graph, \textit{Proc. Amer. Math. Soc.} \textbf{242} (2013), 213--225.
\end{thebibliography}
